\documentclass[secthm,seceqn,amsthm,ussrhead]{amsart}
\usepackage{amsmath,latexsym}
\usepackage[psamsfonts]{amssymb}
\usepackage{times}
\usepackage[mathcal]{euscript}
\numberwithin{equation}{section}

\newcommand{\bbT}{\mathbb T}

\newcommand{\cE}{\mathcal E}

\newcommand{\N}{\mathbb{N}}

\newcommand{\R}{\mathbb{R}}

\newcommand{\T}{\mathbb{T}}

\newcommand{\Z}{\mathbb{Z}}

\newcommand{\be}{\begin{equation}}
\newcommand{\ee}{\end{equation}}

\renewcommand{\Im}{{\ensuremath{\mathrm{Im}}}}

\newcommand{\supp}{{\ensuremath{\mathrm{supp}}}}

{\bf}{\it}
\newtheorem{theorem}{Theorem}[section]

\newtheorem{hypothesis}[theorem]{Hypothesis}
\newtheorem{remark}[theorem]{Remark}
\newtheorem{example}[theorem]{Example}

\large 
\date{\today}

\begin{document}
\large
\title{
On the discrete spectrum of   two-particle discrete Schr\"odinger operators}
\author{Z.~I.~Muminov}

\address{
{ Samarkand State University, University Boulevard 15, 703004,
Samarkand (Uzbekistan)} E-mail:~zimuminov@mail.ru}

\address{$^3$ Samarkand division of Academy of sciences of
Uzbekistan (Uzbekistan) }

\begin{abstract}
 In the present paper our aim  is  to
explore some spectral properties of the family two-particle discrete
Schr\"odinger operators $h^{\mathrm{d}}(k)=h^{\mathrm{d}}_0(k)+ \mathbf{v},$ $k\in \T^\mathrm{d},$ on the $\mathrm{d}$ dimensional lattice $\Z^{\mathrm{d}},$ $\mathrm{d}\geq 1,$ $k$ being the two-particle  quasi-momentum.
Under some condition  in the case $k\in
\T^{\mathrm{d}}\setminus (-\pi,\pi)^{\mathrm{d}},$ we establish
necessary and sufficient conditions for existence of infinite discrete spectrum of the operator $h^{\mathrm{d}}(k)$.
\end{abstract}

\maketitle

\section{Introduction.}

Some spectral properties of the two-particle Schr\"odinger operators
{are studied} in \cite{AbLak,AlLakAb,AbLakUS,ALMM,ALzM,Faria}.

The fundamental difference between multiparticle discrete and
continuous Schr\"o\-din\-ger operators is that in the discrete case the
analogue of the Laplasian $-\Delta$ is not rotationally invariant.

 Due to this fact, the  Hamiltonian of a system does not separate into two parts, one relating to the center-of-motion and other one to the internal degrees of freedom. In particular,  such a handy characteristics of inertia
as mass is not available. Moreover, such a natural local
substituter as the effective mass-tensor (of a ground state)
depends  on the quasi-momentum of the system and, in addition, it
is only semi-additive
 (with respect to the partial order on the set of positive
definite matrices). This is the so-called {\it excess mass}
 phenomenon for lattice systems
 (see, e.g., \cite{Mat} and  \cite{Mog}):
 the effective mass of the bound state
of an $N$-particle system is greater than (but, in general,
 not equal to) the sum of the effective
masses of the constituent quasi-particles.

The two-particle problem on lattices can be reduced to an effective
one-particle problem by using the Gelfand transform instead: the underlying Hilbert space
$\ell^2((\Z^\mathrm{d})^2)$  is decomposed {to a} direct
von Neumann integral, associated with the representation of the
discrete group $\Z^\mathrm{d}$  by the shift operators on the
lattice and then the total two-body Hamiltonian appears to be
decomposable as well. In contrast to the continuous case, the
corresponding  fiber Hamiltonians $h^{\mathrm{d}}(k) $ associated
with  the direct decomposition depend parametrically on the internal
binding $k$, the quasi-momentum, which  ranges over a cell of the
dual lattice. As a consequence, due to the loss of the spherical
symmetry of the problem,  the spectra of the family
$h^{\mathrm{d}}(k)$ turn out to be  rather sensitive to the
variation of the quasi-momentum $k,$ $k\in \T^\mathrm{d}\equiv
(-\pi,\pi]^\mathrm{d}$.

For example, in \cite{AbLak,ALzM,Faria} it is established that the
discrete Schr\"odinger operator $h^3(k),$ with a zero-range
attractive potential, for all non-zero values of the quasi-momentum
$0\ne k \in \T^3,$ has a unique eigenvalue below the essential
spectrum if  $h^{\mathrm{d}}(0)$  has a virtual level (see
\cite{AbLak,ALzM}) or an eigenvalue below its essential spectrum
(see  \cite{Faria}).

The similar result is obtained in  \cite{ALMM}, which is the
(variational) proof of  existence of the discrete spectrum  below
the bottom of the essential spectrum of the fiber Hamiltonians
$h^{\mathrm{d}}(k)$  for all non-zero values of the quasi-momentum
$0\ne k \in \T^{\mathrm{d}},$ provided that  the Hamiltonian
 $h^\mathrm{d}(0)$  has either a virtual level
 (in dimensions  of three and four) or a
 threshold eigenvalue
 (in all dimensions $\mathrm{d}\ge 5$).
 We emphasize that the authors of \cite{ALMM}
  considered   more general class of
  two-particle discrete Hamiltonians interacting via
a   short-range {pair}  potentials.

  It is also worth to mention that  two-particle discrete
 operator $h^\mathrm{3}(k),$  for some
values of the quasi-momentum $k\ne 0$  may generate a ``rich''
infinite discrete   spectrum outside its essential spectrum and
these eigenvalues {accumulate} at the
edge of its essential spectrum (see \cite{AlLakAb}).

In this paper, we explore the some spectral properties of some
$\mathrm{d}$-dimensional two-particle discrete Schr\"odinger operator
$h^\mathrm{d}(k)=h^\mathrm{d}_0(k)+\mathbf{v}, \, k\in \T^{\mathrm{d}}$, $k$ being
 the two-particle quasi-momentum (see \eqref{h}), corresponding to the energy operator
$H^{\mathrm{d}}$ of the system of two quantum particles moving on
the $\mathrm{d}$-dimensional lattice $\mathbb{Z}^{\mathrm{d}}$  with the
short-range potential $\mathbf{v}$. In the case $k\in
\T^{\mathrm{d}}\setminus (-\pi,\pi)^{\mathrm{d}},$ we establish
necessary and sufficient conditions for existence of infinite number
eigenvalues of the operator $h^{\mathrm{d}}(k)$ (Theorem \ref{main3}).

 We now describe the organization of the present
paper. In Sec. 2 describe  the two-particle Hamiltonians  in  the coordinate representation, introduce the two-particle
quasi-momentum, and decompose the energy operator into the von
Neumann direct integral of the  fiber Hamiltonians $h^{\mathrm{d}}(k)$, thus
providing the reduction to the effective one-particle case

In Section 3 in case the particles have (no) equal masses
we obtain  the finiteness of the discrete spectrum of
$h^{\mathrm{d}}(k)$  for all $( k \in \T^{\mathrm{d}})$
$k\in \T^{\mathrm{d}}\setminus (-\pi,\pi)^{\mathrm{d}}$,
 provided that   the Hamiltonian
 $h^\mathrm{d}(0)$  has finite discrete spectrum (Theorem \ref{h(0)<h(k)})
 and  prove
the main result of this  paper, Theorem \ref{main3}, in the case
$k\in \T^{\mathrm{d}}\setminus (-\pi,\pi)^{\mathrm{d}}$.

In Appendix \ref{Appendix I}, for readers convenience, in the case $k\in M=\T^{\mathrm{d}}\setminus (-\pi,\pi)^{\mathrm{d}}$  we give an example that
  the infiniteness of the discrete spectrum of $h^{\mathrm{d}}(k)$
 depends on $k\in M$.

In order to facilitate a description of the content of this paper,
we briefly introduce the notation used throughout this manuscript.
Let $\Z^\mathrm{d}$ be the $\mathrm{d}$-dimensional lattice and
$\T^{\mathrm{d}}$ be the $\mathrm{d}$-dimensional torus, the cube
$(-\pi,\pi]^{\mathrm{d}}$ with appropriately  identified sides,
$\mathrm{d}\ge 1.$ We remark that the torus $\T^{\mathrm{d}}$ will always be considered as an abelian
group with respect to the addition and multiplication by the real
numbers { regarded as} operations on $\R^{\mathrm{d}}$ modulo $(2\pi \Z)^{\mathrm{d}}$.
Denote by $L^2((\T^{\mathrm{d}})^m)$ the Hilbert space of square-integrable
functions defined on  $(\T^{\mathrm{d}})^m$.

\section{The description of the two-particle operator}
The free Hamiltonian $ H^\mathrm{d}_{0}$ of two quantum particles
on the $\mathrm{d}$-dimensional lattice
 $\mathbb{Z}^{\mathrm{d}}$ usually associated with the following
self-adjoint operator in the Hilbert space
 $\ell^2((\Z^{\mathrm{d}})^2)$ of $\ell^2$-sequences $f(x),x\in (\Z^{\mathrm{d}})^2$:
 \begin{equation*}
  H^\mathrm{d}_{0}=-\frac{1}{2m_1}\Delta_{x_1}-\frac{1}{2m_2}\Delta_{x_2},
\end{equation*}
with
\begin{equation*}
\Delta_{x_1}={\Delta}\otimes I_\mathrm{d}, \quad
\Delta_{x_2}=I_\mathrm{d} \otimes{\Delta} ,
\end{equation*}
where $I_\mathrm{d}$ is the identical operator on
$\ell^2(\Z^{\mathrm{d}})$, $m_1, m_2>0$  are meaning
of masses of particles and $\Delta$ is standard discrete Laplasian:
\begin{equation*}
    (\Delta f)(x)= \frac{1}{2}
    \sum_{j=1}^\mathrm{d}
(2f(x)-f(x+\vec{e}_j)-f(x-\vec{e}_j)),\quad f\in\ell _2(\Z^{\mathrm{d}}),
\end{equation*}
that is,
\begin{equation*}\label{Lap}
    \Delta=\frac{1}{2} \sum_{j=1}^\mathrm{d}
(2I_\mathrm{d}-T(\vec{e}_j)-T^*(\vec{e}_j)).
\end{equation*}

Here $T(y)$ is the shift operator by $y\in \Z^{\mathrm{d}}$:
$$
(T(y)f)(x)=f(x+y),\quad f\in\ell _2(\Z^{\mathrm{d}}),
$$
 and
 $\vec{e}_j,$ $j=1,\dots,\mathrm{d},$ is  the unit vector along the $j$-th direction of $\Z^{\mathrm{d}}$.

\begin{remark}
  For more general definitions of discrete  Laplasians see \cite{ALMM},\cite{Yaf1}.
\end{remark}
\begin{remark}\label{rem}
Note that the free Hamiltonian $ H^\mathrm{d}_{0}$ is a
multi-dimensional Laurent-Toeplitz
 type operator defined
 by the function $E(\cdot,\cdot):(\T^\mathrm{d})^2\to \R^1$:
  $$
 E(k_1,k_2)=\frac{1}{m_1}\varepsilon(k_1)+
 \frac{1}{m_2}\varepsilon(k_2),\quad k_1,k_2\in \T^{\mathrm{d}},
 $$
 where
 \begin{equation}\label{cos}
 \varepsilon(p)=\sum_{i=1}^\mathrm{d}(1-\cos p^{(i)}),\,p=(p^{(1)},\dots,p^{(\mathrm{d})})\in \T^{\mathrm{d}}.
 \end{equation}
\end{remark}

 { The last claim can
be obtained by the Fourier transform $
{\mathcal{F}}:\ell^2(({\Z}^{\mathrm{d}})^2)\rightarrow
L^2((\T^{\mathrm{d}})^2) $ which gives the 
unitarity of  $
 H^\mathrm{d}_{0}$ and the multiplication operator
 on $L^2((\T^{\mathrm{d}})^2)$ by the function $E(\cdot,\cdot).$}
It is easy to see that the free Hamiltonian $ H^\mathrm{d}_0$ is a
bounded operator and its spectrum coincides with the interval
$[0,\,\mathrm{d}(m_1^{-1}+m_2^{-1})].$

The total Hamiltonian  $  H^\mathrm{d}$ (in the coordinate
representation) of the system of the two quantum particles  moving on
$\mathrm{d}$-dimensional lattice $\mathbb{Z}^{\mathrm{d}}$ with
the real-valued pair interaction $ V$ is a self-adjoint bounded
operator in the Hilbert space $\ell^2((\Z^{\mathrm{d}})^2)$ of
the form
\begin{equation*}
 H^\mathrm{d} = H^\mathrm{d}_{0}+ {V},
 \end{equation*}
where \begin{equation*}
({V} f)(x_1,x_2) ={{v}(x_1
-x_2)f(x_1,x_2)},\quad f \in \ell^2((\Z^{\mathrm{d}})^2),
\end{equation*}
with $ v:\Z^{\mathrm{d}}\to \R$ a bounded function.

As we note above the Hamiltonian $ H^\mathrm{d}$ does not split into a sum of a
center of mass and relative kinetic energy as in the
continuous case using the center of mass and relative coordinates.

Using the
direct  von Neumann integral decomposition the
 two-particle problem on lattices can be reduced to an effective
one-particle problem. Here we shortly recall some results of \cite{Yaf1} related to direct
integral decomposition.

The energy operator $ H^{\mathrm{d}}$ is obviously commutable with
the group of translations $ U_s$, $s\in \Z^{\mathrm{d}}$:
\begin{equation*}
( U_s  f )(x_1,x_2)=f (x_1+s,x_2+s),\quad f \in
\ell^2((\Z^{\mathrm{d}})^2),\quad x_1, x_2 \in \Z^{\mathrm{d}},
\end{equation*}
for any $s\in \Z^{\mathrm{d}}.$

One can check that the group $ \left\{U_s\right\}_{s\in
\Z^{\mathrm{d}}}$ is the unitary representation of the abelian group
$\Z^{\mathrm{d}}$ in the Hilbert space
$\ell^2(({\Z}^{\mathrm{d}})^2)$.

 { Consequently, we  can decompose the Hilbert space
 $\ell^2((\Z^{\mathrm{d}})^2)$ in a diagonal for
the operator $ H^{\mathrm{d}}$ direct integral whose fiber is
parameterized by $k\in \T^{\mathrm{d}}$ and consists of functions on
$\ell^2(({\Z}^{\mathrm{d}})^2)$ satisfying the condition}
\begin{equation*}\label{ff}
(U_s f)(x_1,x_2)=\exp(-i(s,k))f(x_1,x_2),\quad s\in
\Z^{\mathrm{d}},\,k \in \T^{\mathrm{d}}.
\end{equation*}

Here the parameter $k$ is naturally interpreted  as the \emph{total
quasi-momentum } of the two-particles and is called the two-particle
\emph{quasi-momentum}.

More precisely, let us introduce the  mapping $F:
\ell^2((\Z^{\mathrm{d}})^2)\to \ell^2(\Z^{\mathrm{d}})\otimes
L^2(\T^{\mathrm{d}})$ by the equality
\begin{equation*}\label{F}
    (Ff)(x_1,k)=\frac{1}{(2\pi)^{\frac{\mathrm{d}}{2}}}
    \sum_{s\in \Z^{\mathrm{d}}} e^{-\mathrm{i}(s,k)}f(x_1+s,s),\quad f \in
    \ell^2((\Z^{\mathrm{d}})^2).
\end{equation*}

Remark that $F$ is the unitary  mapping of the space
$\ell^2((\Z^{\mathrm{d}})^2)$ onto $\ell^2(\Z^{\mathrm{d}})\otimes
L^2(\T^{\mathrm{d}})$ and its adjoint
$F^{*}:\ell^2(\Z^{\mathrm{d}})\otimes L^2(\T^{\mathrm{d}}) \to
\ell^2((\Z^{\mathrm{d}})^2) $
 defined  as
\begin{align*}\label{FF}
    (F^{*} g)&(x_1,x_2)
    =\frac{1}{(2\pi)^{\frac{\mathrm{d}}{2}}}
    \int\limits_{k'\in \T^{\mathrm{d}}}
    e^{-\mathrm{i}(x_2,k')}g(x_1-x_2,k')dk',
    \quad g \in
\ell^2(\Z^{\mathrm{d}})\otimes L^2(\T^{\mathrm{d}}).
\end{align*}

 It
follows that $\hat H_0^{\mathrm{d}}=F H_0^{\mathrm{d}} F^{*} $
acts in $\ell^2(\Z^{\mathrm{d}})\otimes L^2(\T^{\mathrm{d}})$ as
a multiplication by the operator function $h_0^{\mathrm{d}}(k),$
$k\in \T^{\mathrm{d}},$ acting in the Hilbert space
$\ell^2(\Z^{\mathrm{d}})$ by
\begin{align*} 
 h_0^{\mathrm{d}}(k)=\frac{1}{2}
  \sum_{j=1}^{\mathrm{d}}
  \big({2}\mu(0)I_\mathrm{d}- \mu(k^{(j)}) T(\vec{e}_j)-\overline{\mu(k^{(j)})} T^*(\vec{e}_j)  \big),
\end{align*}
where
$$
\mu(y)={m_1}^{-1}+{m_2}^{-1}{e^{-\mathrm{i}y}},\quad y\in \T^1.
$$

\begin{remark}\label{rem1}
Note that the operator
 $ h_0^{\mathrm{d}}(k),$ $k\in \T^{\mathrm{d}},$ also is a
multi-dimensional Laurent-Toeplitz type operator defined
 by $$E_k(p)=\frac{1}{m_1}\varepsilon(p)+
 \frac{1}{m_2}\varepsilon(k-p),\quad p\in \T^{\mathrm{d}},
 $$
 where
 $\varepsilon(\cdot)$ is defined by \eqref{cos}.
\end{remark}

Clearly, the operator $V$ commutes with $\left\{U_s\right\}_{s\in
\Z^{\mathrm{d}}}.$ As a result $FV F^{*} $ acts in
$\ell^2(\Z^{\mathrm{d}})\otimes L^2(\T^{\mathrm{d}})$ as a
multiplication operator  by the form
\begin{equation*}\label{V}
(FV F^{*}g)(x,k)= v(x)g(x,k),\quad g \in
\ell^2(\Z^{\mathrm{d}})\otimes L^2(\T^{\mathrm{d}}).
\end{equation*}

Consequently, the operator
 $\hat H^{\mathrm{d}}=F H^{\mathrm{d}} F^{*} $ acts in
$\ell^2(\Z^{\mathrm{d}})\otimes L^2(\T^{\mathrm{d}})$ as a
multiplication by the operator-function $h^{\mathrm{d}}(k),$
$k\in \T^{\mathrm{d}},$ acting in the Hilbert space
$\ell^2(\Z^{\mathrm{d}})$ in the form
\begin{align}\label{h}
 h^{\mathrm{d}}(k)=h_0^{\mathrm{d}}(k)+\mathbf{v},
\end{align}
where $\mathbf{v}$ is the multiplication operator by the function
$v(\cdot)$ in $\ell^2(\Z^{\mathrm{d}}).$

This procedure is quite similar to the "separation" of motion of the
center of mass for dispersive Hamiltonian in the continuous case.
 However, in this case the role of
$\Z^{\mathrm{d}}$ and $\T^{\mathrm{d}}$ was played by
$\R^{\mathrm{d}}.$

The form \eqref{h} for $h^{\mathrm{d}}(k)$ shows that the two-particle problem to be reduced to an effective one-particle problem.

For any $l,$ $l\in N_{\mathrm{d}}=\{1,\dots,\mathrm{d}\},$ let $\mathcal{A}_l^{\mathrm{d}}$ be the family of all ordered  set  ${\boldsymbol{\alpha}}=\{\alpha_1,\dots,\alpha_l\}\subset N_{\mathrm{d}},$ and
$\bar {\boldsymbol{\alpha}}$ be the  complement of $\boldsymbol{\alpha}\in \mathcal{A}_l^{\mathrm{d}}$ with respect to the set $N_\mathrm{d}$.

 We denote by $M^\mathrm{d}_l(\boldsymbol{\alpha})$, $\boldsymbol{\alpha}\in \mathcal{A}_l^{\mathrm{d}},$ the set of all $k\in  M=\T^{\mathrm{d}}\setminus (-\pi,\pi)^{\mathrm{d}} $ such that
only $\alpha_1$-th \dots $\alpha_l$ -th
 coordinates are equal to $\pi$, i.e
\begin{align*}
  &M^\mathrm{d}_l(\boldsymbol\alpha)\equiv \{k\in \T^{\mathrm{d}}:
  k^{(j)}=\pi,\,j\in \boldsymbol\alpha \quad \mbox{and}\quad      k^{(j)}\neq \pi,\,j\in \bar{\boldsymbol\alpha}\}.
\end{align*}

Let be $\Z^{l}({\boldsymbol\alpha})\subset \Z^{\mathrm{d}} $ be the lattice with the basis
$\{ \vec{e}_j\},\,j\in {\boldsymbol\alpha}$. Of course $\Z^{l}({\boldsymbol\alpha})$
is isomorphic to the abelian group $\Z^{l}$ and $\Z^{l}({\boldsymbol\alpha})\oplus \Z^{\mathrm{d}-l}(\bar{\boldsymbol\alpha})=\Z^{\mathrm{d}}.$

 By $\Z^{\mathrm{d}-l}_x(\bar{\boldsymbol\alpha})$ we denote the shift of $\Z^{\mathrm{d}-l}(\bar{\boldsymbol\alpha})$ by $x\in \Z^{l}({\boldsymbol\alpha})$, i.e
$$
\Z^{\mathrm{d}-l}_x(\bar{\boldsymbol\alpha})=\Z^{\mathrm{d}-l}(\bar{\boldsymbol\alpha})+x.
$$

In order to  get the main results of the paper
 we assume the
following technical hypotheses that guarantee  the finiteness of the discrete spectrum of the operator $h^{\mathrm{d}}(0)$,
corresponding to the zero value of the quasi-momentum $k,$  and  the compactness
 of the interaction $\mathbf{v}$.

\begin{hypothesis}\label{hypo}
 Here and further with regard to the function
 $v(\cdot)$ it is assumed that:\\
 (A) $\quad \lim_{x\to \infty}v(x)=0;$\\
 (B)\quad for any $l \in N_{\mathrm{\mathrm{d}-1}},$  $\boldsymbol\alpha \in \mathcal{A}_l^{\mathrm{d}}$, $x\in  \Z^{l}({\boldsymbol\alpha})$ the potential
 $\tilde{{\bf v}}$, corresponding to the restriction $\tilde{v}(\cdot)$ of the function ${v}(\cdot)$ on the set $ \Z^{\mathrm{d}-l}_x(\bar{\boldsymbol\alpha})$,
 do not produce infinite discrete spectrum for the operator
 $$
 h^{\mathrm{d}-l}(0)= h^{\mathrm{d}-l}_0(0)+\tilde{{\bf v}}.
 $$
\end{hypothesis}

\begin{remark}\label{hypoq} The real valued function $v(\cdot)$, being exponentially decreasing function, satisfies Hypothesis \ref{hypo}.
\end{remark}

\section{The formulation and proof of the main results}

Under the  Fourier transform $
{\mathcal{F}}:\ell^2({\Z}^{\mathrm{d}})\rightarrow
L^2(\T^{\mathrm{d}}) $ the operator
 $  h^\mathrm{d}(k),$ $k\in \T^\mathrm{d},$
  turns into the operator $\hat  h^\mathrm{d}(k),$ $k\in \T^\mathrm{d},$
  (the Friedrichs model)
 acting in $L^2(\T^{\mathrm{d}})$ as
 \begin{equation*}\label{ss}
\hat h^\mathrm{d}(k)=\hat h_0^\mathrm{d}(k)+\hat {\mathbf{v}},
    \end{equation*}
    where
    $\hat h_0^\mathrm{d}(k)$ is the multiplication operator by the
    function $E_k (q) $:
\begin{equation*}\label{H0(k)=}
(\hat {h}^{\mathrm{d}}_0(k)f)(q)=E_k (q)f(q),\quad f\in
L^2(\T^{\mathrm{d}}),
\end{equation*}
and $\hat {\mathbf{v}}$ is the integral operator of convolution
type:
\begin{equation*}\label{V}
( \hat {\mathbf{v}}f)(q)=(2\pi
)^{-\frac{\mathrm{d}}{2}}\int\limits_{\T^{\mathrm{d}}} \hat  v
(q-s)f(s)\mathrm{d} s,\quad f\in L^2(\T^{\mathrm{d}}).
 \end{equation*}

Here $\hat  v (\cdot)$ is the Fourier series with the Fourier
coefficients $v(s), s\in \Z^{\mathrm{d}}.$

Under assumption $(A)$ of  Hypothesis \ref{hypo} the perturbation $\mathbf{v}$ of the
operator $h^{\mathrm{d}}_0(k)$, $k\in\T^{\mathrm{d}}$,  is a
Hilbert-Schmidt operator and, therefore, according to the Weil
theorem, the essential  spectrum of the operator $h^{\mathrm{d}}(k)$
fills in the following interval on the real axis:
$$
\sigma_{\text{ess}}(h(k))= [{ \cE}_{\min}(k) ,{\cE }_{\max}(k)],
$$
where
 $$
{\cE }_{\min}(k)=\min_{q{\in }{\bbT}^{\mathrm{d}}}E_{k} (q),\quad
{\cE }_{\max} (k)=\max_{q{\in }{\bbT}^{\mathrm{d}}} E_{k} (q).
 $$


 The function $ E_k(p)$ can be rewritten in the form
\begin{equation}\label{echiladigan}
E_k(p)=\mathrm{d} \mu(0)-\sum_{j=1}^\mathrm{d} r(k^{(j)})\cos
(p^{(j)}- p(k^{(j)})),
\end{equation}
where the coefficients $r(k^{(j)})$ and $p(k^{(j)})$
are given by
$$
r(k^{(j)})=|\mu(k^{(j)})|,\quad
p(k^{(j)})=\arcsin\frac{\Im \mu(k^{(j)})}{r(k^{(j)})},\quad
k^{(j)}\in (-\pi,\pi].
$$

The equality \eqref{echiladigan} implies the following
representation for $  E_k(\cdot)$:
\begin{equation}\label{represent}
E_k(p+p(k))=\mathrm{d} \mu(0)-\sum_{j=1}^\mathrm{d} r(k^{(j)})\cos
p^{(j)} ,
\end{equation}
where
 $p(k)$ is the
vector-function
\begin{align*}
 p:\T^3\rightarrow \T^3,\,\,\,
p(k)=
 (p(k^{(1)}),\dots,p(k^{(\mathrm{d})}))\in \T^3.
\end{align*}

Then
$$
{\cE }_{\min}(k)=\mathrm{d} \mu(0)-\sum_{j=1}^\mathrm{d} r(k^{(j)}),\quad
{\cE }_{\max}(k)=\mathrm{d} \mu(0)+\sum_{j=1}^\mathrm{d} r(k^{(j)})
 $$
and
\begin{equation}\label{E ineq}
E_k(p+p(k))-{\cE }_{\min}(k)\leq E_0(p)\leq
\frac{1}{A(k)}\big ( E_k(p+p(k))-{\cE }_{\min}(k)
\big),
\end{equation}
since $r(k^{(j)})\leq r(0)$ and $ r(0)\leq  r(k^{(j)})/A(k)$,
where
\begin{equation*}\label{INf}
    A(k)=\min_{j=1,\dots,\mathrm{d}}r(k^{(j)})/r(0).
\end{equation*}
\begin{remark}
  Note that $A(k)=0$ holds and only if
$k\in \T^{\mathrm{d}}\setminus (-\pi,\pi)^{\mathrm{d}}$ and $m_1= m_2$.
\end{remark}

\subsection{ Case $A(k)\neq 0$}
\begin{theorem}\label{h(0)<h(k)}
Assume Hypothesis \ref{hypo} and for  $k\in \T^{\mathrm{d}}$ the inequality $A(k)\neq 0$ holds.  Then the operator $h^{\mathrm{d}}(k)$ has only finite discrete spectrum.
 \end{theorem}
 \begin{proof}
Let be $U_kf(p)=f(p+p(k))$ is the shift operator on $\T^{\mathrm{d}}$ and $A(k)\neq 0$. Using the inequality \eqref{E ineq} one can get the inequalities
\begin{equation*}
    U_k^*\big( \hat h^{\mathrm{d}}_0(k) -{\cE }_{\min}(k)I_{\mathrm{d}}\big)U_k\leq
    \hat h^{\mathrm{d}}_0(0)\leq \frac{1}{A(k)}U_k^* \big (
    \hat h^{\mathrm{d}}_0(k)-{\cE }_{\min}(k)I_{\mathrm{d}}
     \big )U_k.
    \end{equation*}
Consequently, we obtain
\begin{equation}\label{INeq}
    U_k^*  \big( h^{\mathrm{d}}(k)-{\cE }_{\min}(k)I_{\mathrm{d}}\big)U_k\leq
    \hat h^{\mathrm{d}}(0)\leq \frac{1}{A(k)} U_k^*\big (
     \hat h_0^{\mathrm{d}}(k) + A(k)\mathbf{\hat{v}}-{\cE }_{\min}(k)I_{\mathrm{d}}
     \big )U_k,
    \end{equation}
since $U_k^*\mathbf{v}U_k=\mathbf{v}$.

Now we define a unitary
staggering transformation $U$ (see \cite{Faria}), which plays a key role in
understanding the relation between the parts of the discrete spectrum  $\sigma_{disc}(h^{\mathrm{d}}(k))$
lying below and above $\sigma_{ess}(h^{\mathrm{d}}(k))$ of $(h^{\mathrm{d}}(k))$
$$
(Uf)(x)=(-1)^{\sum_{j=1}^{\mathrm{d}} x^{(j)}}f(x),\quad f \in \ell^2(\Z^{\mathrm{d}}).
$$

The transformation $U$ has the important intertwining property
\begin{equation*}\label{Uh}
    h^{\mathrm{d}}_0(k)+\mathbf{v}=U^{-1}\big( -1(4\mathrm{d}(m_1^{-1}+m_2^{-1})+h^{\mathrm{d}}_0(k)-\mathbf{v}
        \big )U
\end{equation*}

This equality shows that it is enough to consider the discrete spectrum below the essential spectrum of $h^{\mathrm{d}}_0(k)+\mathbf{v}$.

This facts,  inequalities \ref{INeq},  Hypothesis \ref{hypo} and  the min-max principle yields the statement of the theorem.
\end{proof}

To get oriented, let us recall some results and facts
 relevant to our discussion.

In one- and two- dimensional case the following slightly changed form of Theorem 4.5 in \cite{DHKS} hold and  its proof is omitted.
\begin{theorem}\label{DSKH} Assume Hypothesis \ref{hypo} and $A(k)\neq 0$.
 Let $\mathrm{d} = 1$ or $ 2$. If $\sigma(h^\mathrm{d}(k))\subset [{\cE }_{\min}(k),{\cE }_{\max}(k)] $, then $\mathbf{v} = 0$.
\end{theorem}

 This is connected to the fact that the discrete and continuous one-particle Schr\"odinger operators in one and two dimensions always have a bound state for nontrivial  potentials (see relevant discussions \cite{AbLakUS}, \cite{AlLakAb} and \cite{Kl},\cite{LL},\cite{BS1}).

{In three and upper dimensional case} for the potentials satisfying Hypothesis \ref{hypo} in
$\mathrm{d}\geq 3$-dimensional case
 two-particle discrete Schr\"odinger operators
$h^{\mathrm{d}}(k),$  has only finite number  of eigenvalues outside
the essential spectrum. Moreover, for sufficiently small
potentials, the corresponding operator need not have bound states (see \cite{AlLakAb}).

This fact follows  from the uniformly boundedness of the norm of the
Birman-Schwinger compact operator $G(k,z)=
-({r_0^{\mathrm{d}}}(k,z))^{\frac{1}{2}}
\mathbf{v}({r_0^{\mathrm{d}}}(k,z))^{\frac{1}{2}},$  $r_0^{\mathrm{d}}(k,z)=(
h_0^{\mathrm{d}}(k)-zI_\mathrm{d}
)^{-1},$ $z<\cE_{min}(k),$ near
the bottom of the essential spectrum of $h^\mathrm{d}(k).$ Moreover,
in this case  the discrete spectrum $\sigma_{\mathrm{disc}}(h^{\mathrm{d}}(k))$ of
$h^{\mathrm{d}}(k)$  is empty, since  for sufficiently small
potentials the norm of the Birman-Schwinger operator $G(k,z),$
$z<\cE_{min}(k),$  is small near the bottom of the essential
spectrum of $h^\mathrm{d}(k).$
\subsection{Case $A(k)=0$}
 The equality $A(k)=0$ implies  this equality  $m_1=m_2$ and $k\in M=\T^{\mathrm{d}}\setminus (-\pi,\pi)^{\mathrm{d}}$.

For any $\{j\}\in \mathcal{A}_1^{\mathrm{d}}$, $\boldsymbol\alpha \in \mathcal{A}_l^{\mathrm{d}}$, $l\in N_\mathrm{d}$, and $n\in \N$
 we say the set
$$
\Pi_n^{\mathcal{\mathrm{d}}-1}({{\{j\}}})=\bigcup_{ x\in\{-n,\dots,n\}}\Z_{x}^{\mathrm{d}-1}(\overline{{\{j\}}}),
$$
 resp.
 $$
\Pi_n^{\mathrm{d}-l}({\boldsymbol\alpha})= \bigcap_{\alpha_j \in \boldsymbol\alpha}
\Pi_n^{\mathrm{d}-1}({{\{\alpha_j\}}}),
$$
is the $\mathrm{d}-1$ resp. $\mathrm{d}-l$ dimensional lattice strip on $\Z^{\mathrm{d}}$, along  direction  corresponding
to $\overline{\{j\}}$ resp. $\bar{\boldsymbol\alpha}$.

Define
$$
X_n({\boldsymbol\alpha})=\{x\in \Z^{l}({\boldsymbol\alpha}):\quad x^{(j )}\in \{-n,\dots,n\}\quad \mbox{for all}\quad j\in  {\boldsymbol\alpha} \},\quad n\in \N^1,
$$
the subset (the lattice cube with side $n$) of $\Z^{l}({\boldsymbol\alpha}).$

One can check that for any $\boldsymbol\alpha \in \mathcal{A}_l^{\mathrm{d}}$, $l\in N_{\mathrm{d}}$,
$$
\Pi_n^{\mathrm{d}-l}({\boldsymbol\alpha})=
\bigcup_{\hat{x}\in
X_n({\boldsymbol\alpha})}
\Z^{\mathrm{d}-l}_{\hat x}(\bar{\boldsymbol\alpha})\quad \mbox{and}\quad
\Pi_n^{\mathrm{d}-l}({\boldsymbol\alpha})=
\Z^{\mathrm{d}-l}(\bar{\boldsymbol\alpha}) \oplus X_n({\boldsymbol\alpha})
$$
and in particular case $l=\mathrm{d}$ the equality $\Pi_n^{\mathrm{d}-l}({\boldsymbol\alpha})=X_n({\boldsymbol\alpha})$ holds and
$\Pi_n^{\mathrm{d}-l}({\boldsymbol\alpha})$ is the lattice cube with side $n$ in $\Z^{\mathrm{d}}$.

\begin{theorem}\label{main3} Assume Hypothesis \ref{hypo}. The for any $l\in N_{\mathrm{d}}$,
 $\boldsymbol\alpha \in \mathcal{A}_l^{\mathrm{d}}$,  and for all
$k\in M_{l}^\mathrm{d}(\boldsymbol\alpha)$
we have:\\
a) if $\mathrm{d}-2\leq l\leq
\mathrm{d}$ then the
operator $h^{\mathrm{d}}(k)$ has an infinite number of eigenvalues
outside of the essential spectrum if and only if for any $n\in \N$ the
relation
$$\supp  v \nsubseteq \Pi_n^{\mathrm{d}-l}({\boldsymbol\alpha})
 $$
is  accomplished.\\
b) if $1\leq l<\mathrm{d}-2$ then the
operator $h^{\mathrm{d}}(k)$ has a finite number of eigenvalues
outside of the essential spectrum.
\end{theorem}
\begin{proof}
To be simplicity  we take $\bar{\boldsymbol\alpha}=\{1,\dots, \mathrm{d}-l\}$,
${\boldsymbol\alpha}=\{\mathrm{d}-l+1,\dots, \mathrm{d}\}$
and we
use notations $\tilde k=(k^{(1)},\dots,k^{(\mathrm{d}-l)}) \in
\T^{\mathrm{d}-l}$ for $k=(\tilde{k},\pi,\dots,\pi)
\in  M_{l}^\mathrm{d}(\boldsymbol\alpha),$  and
$x=(\tilde{x},\hat{x})\in \Z^{\mathrm{d}}$ for $\tilde{x}\in \Z^{\mathrm{d}-l}(\bar{\boldsymbol\alpha}),$
$\,\hat{x}\in \Z^{l}({\boldsymbol\alpha})$.

Clearly
\begin{equation*}\label{direct}
\ell^2(\Z^\mathrm{d})=
\bigcup_{\hat{x}\in \Z^{l} }\oplus\ell^2(\Z^{\mathrm{d}-l}_{\hat x}(\bar{\boldsymbol\alpha}))
\end{equation*}
and $\ell^2(\Z^{\mathrm{d}-l}_{\hat{x}}(\bar{\boldsymbol\alpha})),$ $\hat{x}\in \Z^l,$ is isomorphic to
$\ell^2(\Z^{\mathrm{d}-l }).$ This isomorphism is established with the
restriction $P_{\hat{x}}\equiv P\big |_{\ell^2(\Z^{\mathrm{d}-l}_{\hat{x}}(\bar{\boldsymbol\alpha})) }$ of the
isometric operator $P:\ell^2(\Z^{\mathrm{d}})\to \ell^2(\Z^{\mathrm{d}-l}_{\hat{x}}(\bar{\boldsymbol\alpha}))$:
\begin{align*}
&(P f)(\tilde{x})= f(\tilde{x},\hat{x}), \quad  f \in \ell^2(\Z^{\mathrm{d}}),
\end{align*}
on  $\ell^2(\Z^{\mathrm{d}-l}_{\hat{x}}(\bar{\boldsymbol\alpha}))$.

If $k \in M_{l}^\mathrm{d}(\boldsymbol\alpha)$, then
\begin{align*} \label{h0(k)}
 h_0^{\mathrm{d}}(k)=\frac{1}{2}
  \sum_{j=1}^{\mathrm{d}-l}
  \big({2}\mu(0)I_{\mathrm{d}}- \mu(k^{(j)}) T(\vec{e}_j)-\overline{\mu(k^{(j)})} T^*(\vec{e}_j)  \big)
  +l\mu(0)I_{\mathrm{d}}.
\end{align*}

Using the last equality it is easy to see that
 the Hilbert space $\ell^2(\Z^{\mathrm{d}-l}_{\hat{x}}(\bar{\boldsymbol\alpha}))$ is
invariant under $h^\mathrm{d}_0(k)$ and $\mathbf{v}.$

One can
check that
 $ h^{\mathrm{d}}(k)$ is described as
\begin{equation}\label{deric.oper}
 h^{\mathrm{d}}(k)=\sum_{x\in \Z^l} \oplus P_{\hat{x}}^*(l\mu(0)I_{\mathrm{d}-l}+h_{\hat{x}}^{\mathrm{d}-l}(\tilde{k}))
P_{\hat{x}},
\end{equation}
where
$h_{\hat{x}}^{\mathrm{d}-l}(\tilde{k})$ is the $\mathrm{d}-l$-dimensional discrete
Schr\"odinger operator corresponding  to the Hamiltonian of a
system of two identical particles  moving on  $\Z^{\mathrm{d}-l}$ and
interacting short-range attractive potential ${\mathbf{v}}_{\hat{x}},$ being
multiplication operator by the restriction ${v}_{\hat{x}}(\cdot),$
${v}_{\hat{x}}:{\Z}^{\mathrm{d}-l}(\bar{\boldsymbol\alpha})\to \R^1, $ of the function $v(\cdot)$ on  ${\Z}^{\mathrm{d}-l}_{\hat{x}}(\bar{\boldsymbol\alpha})\simeq {\Z}^{\mathrm{d}-l}.$

That is  we decompose $h^{\mathrm{d}}(k)$
a direct sum of $\mathrm{d}-l$ dimensional two-particle
Schr\"odinger operators
$$
h^{\mathrm{d}-l}_{\hat{x}}(\tilde{k})=h^{\mathrm{d}-l}_0(\tilde{k})+
\mathbf{v}_{\hat{x}},\,
 \tilde k\in \T^{\mathrm{d}-l},\,
 \hat{x}\in \Z^l,
 $$
 with  the
potential $\mathbf{v}_{\hat{x}},$ $\hat{x}\in \Z^{l}:$
\begin{equation}\label{deric.oper1}
h^{\mathrm{d}}(k)\simeq \sum_{\hat{x}\in \Z^l}
 \oplus
\big
(l\mu(0)I_{\mathrm{d}-l}+h_0^{\mathrm{d}-l}(\tilde{k})+\mathbf{v}_{\hat{x}}\big
).
\end{equation}

Since \eqref{deric.oper} and  $h_0^{\mathrm{d}-l}(\tilde k)$ does not depend
on $x$ and $\mathbf{v}_{\hat{x}}$ is a compact operator for the essential
spectrum of $h^{\mathrm{d}}(k)$ we  get
\begin{equation*}\label{spectr0}
    \sigma_{ess}\big ( h^{\mathrm{d}}(k)\big )=l\mu(0)+\bigcup_{x\in \Z^l}
    \sigma_{ess}\big ( h^{\mathrm{d}-l}_{\hat{x}}(\tilde{k})\big )=
    [\cE_{min}(k),\cE_{max}(k)]
\end{equation*}
and hence
\begin{equation}\label{spectr}
    \sigma_{disc}\big ( h^{\mathrm{d}}(k)\big )=l\mu(0)+\bigcup_{x\in \Z^l}
    \sigma_{disc}\big ( h^{\mathrm{d}-l}_{\hat{x}}(\tilde{k})\big ).
\end{equation}

a) Case $l=\mathrm{d}.$ Then $M_{\mathrm{d}}^\mathrm{d}(\boldsymbol\alpha)=\{\vec{\pi} \},$ $\vec{\pi}=(\pi,\dots,\pi)\in \T^{\mathrm{d}}$ and
$h^{\mathrm{d}}(\vec{\pi})=\mathrm{d}\mu(0)+\mathbf{v}.$
Hence $\sigma_{disc}\big ( h^{\mathrm{d}}(\vec{\pi})\big )=\mathrm{d}\mu(0)+
\cup_{x\in \Z^{\mathrm{d}}}\{ v(x)\}.
$

Consequently, $\sigma_{disc}\big ( h^{\mathrm{d}}(\vec{\pi})\big )$ is infinite
iff for any $n\in \N^n$ the  support  $\mathrm{supp} v  $ of $v(\cdot)$ does not belong in
$\Pi_n^{0}({\boldsymbol\alpha})$-the lattice cube on $\Z^{\mathrm{d}}$ with side $2n$.

Case  $\mathrm{d}-l=1$ or $=2$.  Then by Theorem \ref{DSKH} the operator
 $\sigma_{disc}( h^{\mathrm{d}-l}_{\hat{x}}(\tilde{k}))$ is  nonempty  if and only if the
 potential  $\mathbf{v}_{\hat{x}}$ is nonzero.

So by \eqref{spectr} the equality
$ \sharp \sigma_{disc}\big ( h^\mathrm{d}(k)\big )=\infty$ holds iff
for any $n\in \N$ there exists $\hat x \in X_n({\boldsymbol\alpha})$ such that
$\mathbf{v}_{\hat{x}}\neq 0 $, that is
$\mathrm{supp} v  $  does not belong in
$$
\Pi_n^{\mathrm{d}-l}({\boldsymbol\alpha})=
\bigcup_{\hat{x}\in
X_n({\boldsymbol\alpha})}
\Z^{\mathrm{d}-l}_{\hat x}(\bar{\boldsymbol\alpha}).
$$

b)  If $\mathrm{d}-l=3$ or $\mathrm{d}-l>3$ (that is $1\leq l< \mathrm{d}-2$) then by Theorem \ref{h(0)<h(k)}  the operator
 $ h^{\mathrm{d}-l}_{\hat{x}}(\tilde{k}),\,
 \tilde k\in \T^{\mathrm{d}-l},\,
 \hat{x}\in \Z^l,$  has no the infinite number  of eigenvalues
outside of the essential spectrum.

By virtue Hypothesis \ref{hypo} for any $\epsilon>0$ there exists $n\in \N^1$ such that for all $x\in  \bigcup_{-n\leq x\leq
n} \Z_{i,x}^\mathrm{d} $ the inequality
 $|v(x)|<\epsilon$ holds.

Then from the discussed facts  relating to
 the {emptiness} of the  (three and upper dimensional )
 two-particle discrete operators there exists $n\in \N^1$ such that for all $\hat{x} \in
  \Z^{\mathrm{d}} ({\boldsymbol\alpha})\setminus X_n^{\mathrm{d}-l}({\boldsymbol\alpha})
  $ the discrete spectr
  $\sigma_{disc}(h^{\mathrm{d}-l}_{\hat{x}}(\tilde{k}))$ is empty.

Hence
\begin{equation}\label{spectr}
    \sigma_{disc}\big ( h^{\mathrm{d}}(k)\big )=l\mu(0)+\bigcup_{x\in X_n^{\mathrm{d}-l}({\boldsymbol\alpha})}
    \sigma_{disc}\big ( h^{\mathrm{d}-l}_{\hat{x}}(\tilde{k})\big ).
\end{equation}
since \eqref{spectr}.

This fact and compactness of $ X_n^{\mathrm{d}-l}({\boldsymbol\alpha})$ ended the proof of Theorem \ref{main3}.
\end{proof}

\begin{remark}
In the case $k\in M^\mathrm{d}(\alpha_1,\dots,\alpha_l) $ the
Fourier symbol $E_k(\cdot)$ of the free one-particle discrete
operator $h_0^{\mathrm{d}}(k)$ attains its minimum along a $l$ dimensional
manifold homeomorphic to $\T^l.$ If this manifold is $\mathrm{d}-1$ or $\mathrm{d}$
dimensional then there exists potential, such that, generates
infinite discrete spectrum for the operator $h_0^{\mathrm{d}}(k).$

This result can  resemble with the result established in
\cite{Pan}. The authors of this paper provided an elemantary
proof that negative potential lead to infinite discrete spectrum
for the perturbation of Hamiltonian whose Fourier symbol attains
its minimal value on $n-1$ dimensional manifold of $\R^n.$
\end{remark}

\appendix \label{Appendix I}
\section{Example of the infiniteness of the discrete spectrum}

The main goal of this appendix is to show that
 in the case  $k\in \T^{\mathrm{d}}\setminus (-\pi,\pi)^{\mathrm{d}}$ the infiniteness of the discrete spectrum depends on the direction $\boldsymbol\alpha\in \mathcal{A}^{\mathrm{d}}_l.$
\begin{example} Let $\mathrm{d}=2$  and
$
v(x)=\left\{
       \begin{array}{ll}
         e^{-|x^{(1)}|}, &  x\in \Z^1\times \{0\} \\
         0, & \hbox{otherwise.}
       \end{array}
     \right.
$
\end{example}

 Then for $\boldsymbol\alpha=\{1\}\in \mathcal{A}^2_1$,  $n\in \N^1$,
we have
$$
\Pi_n^{\mathrm{d}-1}(\{1\})=\{x\in \Z^1: |x^{(1)}|\leq n\},\,
\Pi_n^{\mathrm{d}-1}(\{2\})=\{x\in \Z^1: |x^{(2)}|\leq n\},
$$
and obviously
$$
\supp v(\cdot) \nsubseteq \Pi_n^{\mathrm{d}-1}(\{1\})\quad \mbox{and}\quad      \supp v(\cdot) \subset
\Pi_n^{\mathrm{d}-1}(\{2\}).
$$

Then by theorem for any $k\in M_1^2(\{1\})$
the operator
 $h^{\mathrm{d}}(k)$ has a infinite number of eigenvalues
outside of the essential spectrum, but for $k\in M_2^2(\{2\})$
it has no infinite discrete spectrum.

{\bf Acknowledgments} I thank Prof. S.~N.~Lakaev   for useful discussions and
 gratefully acknowledge the hospitality of the The Abdus Salam International Centre for Theoretical Physics and I am also  indebted to the  the International Mathematical Union for the travel grant.
This work was also partially supported by the
Fundamental Science Foundation of Uzbekistan.


\begin{thebibliography}{99}
\bibitem{AbLak}  {\sc J.~I.~Abdullaev, S.~N.~Lakaev }
{Asymptotics of the Discrete Spectrum of the Three-Particle Schrödinger Difference Operator on a Lattice,} Theoretical and Mathematical Physics,   \textbf{136}, No. 2, (2003)  1096-1109.
\bibitem{AbLakUS} {\sc J.~I.~ Abdullaev, S.~N.~Lakaev: } On  the  Spectral  Properties  of the
Matrix-Valued Friedrichs   Model.  Advances  in  soviet
Mathematics.    American   Mathematical  Society. \textbf{5} (1991)
1-37.
\bibitem{AlLakAb}  {\sc S.~Albeverio, S.~N.~Lakaev} and {\sc J.I.Abdullaev. } On the
spectral properties of two-particle discrete Schrodinger
operators. Preprint, Bonn, Bonn University,  (2004) pp.14.

\bibitem{ALMM}  {\sc S.~Albeverio,  S.~N.~Lakaev, K.~A.~Makarov,
Z.~I.~Muminov:} { The Threshold Effects for the Two-particle
Hamiltonians on Lattices,}
 Comm.Math.Phys. {\bf 262}(2006), 91--115.

\bibitem{ALzM}  {\sc S.~Albeverio, S.~N.~Lakaev,  Z.~I.~Muminov}:
Schr\"{o}dinger
operators on lattices. The Efimov effect and discrete spectrum
asymptotics. Ann.~Henri Poincar\'{e}. {\bf 5} 743--772 (2004).

\bibitem{DHKS} {\sc D.~Damanik, D.~Hundertmark, R.~Killip, B.~Simon}:
  Variational Estimates for Discrete Schr¨odinger Operators
with Potentials of Indefinite Sign.
 Comm.Math.Phys. {\bf 238}(2003), 545--562.

\bibitem{Faria}  {\sc P.~A.~Faria da Viega, L.~Ioriatti and
M.~O'Carrol.} Energy-momentum spectrum of some two-particle
lattice Schr\"{o}dinger Hamiltonian. Phys. Rev. E(3)
 {\bf 66}, 016130, 9 pp. (2002).

\bibitem{Kl} {\sc M.~Klaus}: On the bound state of Schr¨odinger operators in one dimension. Ann. Phys. \textbf{108},  (1977) 288–300.

\bibitem{LL}
{\sc L.~D.~Landau,  E.~M.~Lifshitz}: Quantum Mechanics: Non-relativistic Theory. Course of Theoretical
Physics, Vol. 3. Reading, Mass: Addison-Wesley, 1958.

\bibitem{Mat} {\sc D.~C.~Mattis}:
The few-body problem on a lattice.
 Rev. Modern Phys. {\bf 58},
(1986)  361--379.

\bibitem{Mog} {\sc A.~Mogilner }:
 Hamiltonians in solid state physics as
 multi-particle discrete Schr\"{o}dinger operators: Problems and
 results.
 Advances in Soviet Mathematics {\bf 5},
(1991)  139--194 .

\bibitem{Pan} {\sc K.~Pankrashkin}:
 Variational principle for hamiltonians
with degenerate bottom. In the book I. Beltita, G. Nenciu, R. Purice (Eds.): Mathematical Results in Quantum Mechanics. Proceedings of the QMath10 Conference (World Scientific, 2008) 231-240
Preprint arXiv:0710.4790

\bibitem{BS1}  {\sc  B.~Simon,}
{\it The bound state of weakly coupled Schr\"odinger operators in
one and two-dimentionas,} Ann. Phys. 97, 279-288, (1976). 
\bibitem{RS}  {\sc M.~Reed and  B.~Simon,}
{\it Methods of modern mathematical physics. VI: Analysis of
Operators,}
 Academic Press, New York, 1979.

 \bibitem{Yaf1}  {\sc D.~R.~Yafaev }:
 {\it  Scattering theory: Some old and new problems},
 Lecture Notes in Mathematics, 1735.
 Springer-Verlag, Berlin, 2000, 169 pp.


\end{thebibliography}
\end{document}